\documentclass[12pt, reqno]{amsart}

\usepackage{amssymb,latexsym,amsmath,amsfonts}
\usepackage{enumitem}
\usepackage{mathrsfs}
\usepackage{graphicx}
\usepackage{hyperref}
\usepackage[usenames]{color}
\usepackage{comment}

\DeclareFontFamily{OT1}{rsfs}{}
\DeclareFontShape{OT1}{rsfs}{n}{it}{<-> rsfs10}{}
\DeclareMathAlphabet{\mathscr}{OT1}{rsfs}{n}{it}

\numberwithin{equation}{section}

\theoremstyle{definition}
\newtheorem{definition}{Definition}[section]
\newtheorem{question}[definition]{Question}

\theoremstyle{remark}
\newtheorem{remark}[definition]{Remark}

\theoremstyle{plain}
\newtheorem{theorem}[definition]{Theorem}
\newtheorem{result}[definition]{Result}
\newtheorem{lemma}[definition]{Lemma}
\newtheorem{proposition}[definition]{Proposition}
\newtheorem{corollary}[definition]{Corollary}

\definecolor{DPurple}{rgb}{0.76,0.2,0.69}

\usepackage{color}

\newcommand{\nice}{\Sigma}

\newcommand{\res}{{\rm {Res}}}
\newcommand{\discr}{{\rm {Disc}}}
\newcommand{\trueres}{{\rm {res}}}
\newcommand{\lar}{\mathbb B_r}

\newcommand{\capa}{{\rm cap}}
\newcommand{\homcap}{{\rm{cap}}_h}
\newcommand{\Cancap}{{\rm{cap_{0, \infty}}}}
\newcommand{\gr}{{\rm{g}}}
\newcommand{\ext}{{V}}
\newcommand{\robin}{{\varrho}}
\newcommand{\robcons}{{\gamma_{0, \infty}}}
\newcommand{\ppolar}{{\mathcal E}}
\newcommand{\U}{U_h}
\newcommand{\I}{{I_h}}

\newcommand{\hgtone}{{\rm h}}
\newcommand{\hgttwo}{{\mathscr H}}

\newcommand{\rat}{{r}}

\newcommand{\zero}{\xi}

\newcommand{\sph}{\widehat{\mathbb{C}}}
\newcommand{\C}{\mathbb{C}} 
\newcommand{\R}{\mathbb{R}}
\newcommand{\Z}{\mathbb{Z}}
\newcommand{\N}{\mathbb{N}}

\begin{document}

\title[Equidistribution: algebraic units]{Equidistribution of the conjugates of algebraic units}

\author{Norm Levenberg \and Mayuresh Londhe}
%\thanks{Norm Levenberg is supported by Simons Foundation grant No. 707450}
\address{Department of Mathematics, Indiana University, Bloomington, Indiana 47405, USA}
\email{nlevenbe@iu.edu, mmlondhe@iu.edu}

\begin{abstract}
We prove an equidistribution result for the zeros of polynomials with integer coefficients 
and simple zeros. Specifically, we show that the normalized zero measures associated 
with a sequence of such polynomials, having small height relative to a certain compact set in the complex plane, 
converge to a canonical measure on the set. In particular, this result gives an equidistribution result for the conjugates 
of algebraic units, in the spirit of Bilu's work. Our approach involves lifting these polynomials to
polynomial mappings in two variables and proving an equidistribution result for the normalized 
zero measures in this setting.
\end{abstract}


\maketitle
%\vspace{-0.4cm}

\section{Introduction}\label{S:intro}
This article explores the asymptotic distribution of the zeros of certain polynomials in the complex plane $\C$,
as well as the zeros of certain polynomial mappings in $\C^2$. 
In particular, we investigate the asymptotic distribution of the conjugates of algebraic numbers.
We refer the reader to the recent works \cite{Serre19:dadvpef} and \cite{Smith:aicpd24}
for applications illustrating the relevance of studying the asymptotic distribution 
of the conjugates of algebraic numbers.
One of the initial results in this area was 
established by Bilu \cite{Bilu:limit97}, who showed that for any sequence of 
algebraic numbers $\{\alpha_n\}$ with small Weil height, 
the conjugates $\rm {Conj}(\alpha_n)$ of $\alpha_n$ are equidistributed with respect to the 
normalized arc-length measure on the unit circle. %in $\C$.
Using potential-theoretic methods, Rumely \cite{rumely:bilu99} extended this result to sets with 
logarithmic capacity equal to one.
We refer to \cite{BR:Esp06, CL:mees06, FRL:wq06, Yuan:blbav08} for other equidistribution results.
In this article, we present an equidistribution result for the conjugates of algebraic units. 
Recall that an {\em algebraic unit} is an algebraic integer whose reciprocal is also an algebraic integer. Equivalently, 
$\alpha$ is an algebraic unit if 
the minimal polynomial of $\alpha$ in $\Z[z]$ is monic and has constant term equal to $-1$ or $+1$.%
\smallskip

Let $K$ be a compact subset of $\C \setminus \{0\}$.
In \cite{Cantor:edtdsa80}, D. Cantor introduced a notion of 
capacity of $K$ with respect to $0$ and $\infty$, denoted here by $\Cancap(K)$.
Cantor proved that if every neighborhood of $K$ contains infinitely many complete sets
of conjugate algebraic units, then $\Cancap(K) \geq 1$.
We defer the precise definition of this capacity, as well as other potential-theoretic notions
in this introduction, to Section~\ref{S:prelim}. For now, we provide examples
of sets with $\Cancap(K)=1$ which are of particular interest.
For instance, if $K$ is the preimage of a unit circle or an interval of
length four under a rational map with complex coefficients of the form
\begin{equation}\label{E:monicrat}
\rat (z)= \frac{z^n + \dots \pm 1}{z^j},
\end{equation}
where $j < n$, then $\Cancap(K)=1$ (see Proposition~\ref{P:capacityone}).
In \cite{Robinson:units64}, Robinson provided examples of all real intervals $I$ such that
$\Cancap(I)=1$, although without using the language of capacity.
The interval $[3-2\sqrt2, 3+2\sqrt2]$ is one such example.
\smallskip

Given a polynomial $p(z)=a_dz^d+ \dots + a_0$, we define
the {\em height of $p$ relative to $K$} as
\[
\hgtone_K(p):=\frac{1}{d} \Big(\log \vert a_0a_d\vert + 
\sum_{p(x)=0}[\gr_K (x, 0) +\gr_K (x, \infty)] \Big),
\]
where $\gr_K (z, 0)$ and $\gr_K (z, \infty)$ denote the Green's function of $K$ with pole at $0$ and $\infty$, respectively.
For an algebraic number $\alpha$, we define $\hgtone_K(\alpha):= \hgtone_K(p)$, where $p$ is the minimal
polynomial of $\alpha$ in $\Z[z]$. 
The Green's functions $\gr_K (z, 0)$ and $\gr_K (z, \infty)$ are nonnegative subharmonic functions that
are equal to 0 at most points on $K$ and tend to $\infty$
as $z$ approaches their respective poles.
Loosely speaking, the Green's functions provide a way to quantify
the distance of a point from $K$.
\smallskip

The zeros of polynomials of small height equidistribute with respect to a particular measure
$\nu_K$ on $K$, which we describe after stating the first main result of this article.
\begin{theorem}\label{T:units_equi}
Let $K$ be a compact subset of $\C \setminus \{0\}$ with $\Cancap(K)=1$.
If $\{p_n\}$ is a sequence of polynomials in $\Z[z]$ with simple zeros such that
$\hgtone_K(p_n) \to 0$ and $\deg p_n \to \infty$ as $n \to \infty$,
then the zeros of $p_n$ are equidistributed with respect to $\nu_K$, i.e.,
\[
\nu_n:=\frac{1}{\deg p_n} \sum_{p_n(x)=0} \delta_{x} \to \nu_K \quad \text{as } n \to \infty
\]
in the weak* topology.
\end{theorem}

\noindent We recall that $\nu_n \to \nu_K$ as $n \to \infty$ in the weak* topology if for any bounded continuous function
$f$ on $\C$, we have 
$$\lim_{n \to \infty} \int_{\C} f(z) \, d\nu_n(z) = \int_{\C} f(z) \, d\nu_K(z).$$
We now describe the measure $\nu_K$. Given a compact set $K$ with $\Cancap(K)=1$, let $s_1 \geq 0$ and $s_2 \geq 0$ with $s_1+s_2=1$ satisfy
\[
s_1 \lim_{z \to 0} [\gr_K(z, 0) + \log |z|] +s_2 \gr_K(0, \infty)=0,
\]
\[
s_1\gr_K(\infty, 0)+ s_2 \lim_{z \to \infty} [ \gr_K(z, \infty) - \log |z|]=0.
\]
For instance, for a set that is the preimage of of a unit circle or an interval of
length four under a rational map of the form \eqref{E:monicrat}, we can take
$s_1=j/n$ and $s_2= (n-j)/n$. Then the measure $\nu_K$ that appears in the statement of 
Theorem~\ref{T:units_equi} is given by
\[
\nu_K:= s_1 \nu_K^0+s_2 \nu_K^{\infty},
\]
where $\nu_K^0$
is the harmonic measure of $K$ with respect to $0$ and $\nu_K^{\infty}$ is the equilibrium measure for $K$.
The measure $\nu_K^0$ is the push-forward of the equilibrium measure for $1/K:=\{1/z: z \in K\}$
under the map $z\to 1/z$. If $K=[a, b]$ with $b > a >0$, then
\begin{equation} \label{forserre} 
d\nu_{K}^{0}(t)=\frac{\sqrt{ab}}{\pi t \sqrt{(t - a)(b - t)}} \, dt
\quad \text{and} \quad
d \nu_{K}^{\infty}(t)=\frac{dt}{\pi \sqrt{(t - a)(b - t)}}.
\end{equation}

In particular, Theorem~\ref{T:units_equi} gives an equidistribution result for 
any sequence of 
algebraic units $\{\alpha_n\}$ such that the conjugates of $\alpha_n$ eventually belong to any 
neighborhood of $K$.
Theorem~\ref{T:units_equi} also provides new examples of {\it arithmetic probability measures} as discussed in \cite{OrlSar:ldcai23}.
Observe that the logarithmic capacity of
the interval $[3-2\sqrt2, 3+2\sqrt2]$ is greater than 1, so the equidistribution result in 
\cite{rumely:bilu99} does not apply here. We mention that for many compact sets $K$ with $\Cancap(K)=1$, there do exist sequences of polynomials
satisfying the hypotheses of Theorem~\ref{T:units_equi}; for example, if $K$ is symmetric about the real line,
and 0 and $\infty$ are in the same component of the complement of $K$ such sequences exist. 
We refer to \cite[page 199]{Cantor:edtdsa80} for more on this.

\medskip

{\noindent{\bf Schur--Siegel--Smyth trace problem:}} Recently, A. Smith solved the Schur--Siegel--Smyth trace problem 
\cite{Smith:aicpd24}. He showed there exists a sequence $\{\alpha_n\}$ of totally positive algebraic integers
($\alpha_n$ is a real algebraic integer with $\rm {Conj}(\alpha_n)\subset [0,\infty)$) such that
\[
\liminf_{n \to \infty} {\rm{Trace}} (\alpha_n): =\liminf_{n \to \infty} \sum_{x \in {\rm {Conj}}(\alpha_n)} \frac{ x }{\deg \alpha_n} \approx 1.898.
\]
For this, Smith used a measure constructed by Serre (\cite{AP}, Appendix B), which Smith showed to be a limiting measure of a sequence of algebraic integers.
We demonstrate here how our Theorem~\ref{T:units_equi} sheds more light on their technique.
\smallskip

As mentioned earlier, Robinson \cite{Robinson:units64} gave
examples of all real intervals $I$ such that
$\Cancap(I)=1$.
Let $a$, $b$ and $\tau$ be positive real numbers satisfying 
\[
\log\frac{b-a}{4}= \tau \log \frac{\sqrt{b} + \sqrt{a} }{\sqrt{b} - \sqrt{a} }
 {\text { and }}
\log\frac{b-a}{4ab}=\frac{1}{\tau} \log \frac{\sqrt{b} + \sqrt{a} }{\sqrt{b} - \sqrt{a}}.
\]
The above equations define $a$ and $b$ as continuous, increasing functions of $\tau$ for $\tau > 0$. 
Specifically, $a$ increases from 0 to ${1}/{4}$, and $b$ increases from 4 to $\infty$ as $\tau$ goes from 0 to $\infty$.
For $0 < \tau < \infty$, let $J_\tau$ represent the interval $[a, b]$ determined by these equations. 
It turns out that, for any $0 < \tau < \infty$, we have $\Cancap(J_\tau)=1$.
%moreover, the measure constructed by Serre is the measure $\nu_{J_\tau}$ on $J_\tau$ for some particular $\tau$.
\smallskip

For instance, if $\tau=1/4$, there exists a rational function with real coefficients
$$
\rat(z)= \frac{z^5+ \dots \pm1}{z}
$$
such that $J_{1/4}=\rat^{-1}[-2, 2]$. In this case, the end points of $J:=J_{1/4}$ are given by
\[
\sqrt a=w^5-w^{-3} {\text { and }} \sqrt b=w^5+w^{-3},
\]
where $w$ is the unique real root with $w>1$ of the polynomial $w^{25}-w^9-1$
(see \cite[Section 6]{Robinson:units64}). 
A simple numerical calculation shows that $w \approx 1.03499$. Thus
%\[a  \approx 0.08160168197148374 {\text { and }} b  \approx 4.366418882371485. \]
\[a  \approx 0.08160 {\text { and }} b  \approx 4.36641.\]
Since $\Cancap(J)=1$, there exists a sequence $\{\alpha_n\}$ of algebraic units
such that the conjugates of $\alpha_n$ eventually belong to any 
neighborhood of $J$. Hence $\hgtone_J(\alpha_n)\to 0$ and Theorem~\ref{T:units_equi} 
implies that these conjugates equidistribute with respect to $\nu_J$. By definition of weak* convergence, we have
$
{\rm{Trace}} (\alpha_n)\to \int t \, d\nu_{J}(t)
$
as $n \to \infty$. From Proposition \ref{P:capacityone}, it follows that $\nu_{J}= (1/5) \nu_J^0 + (4/5) \nu_J^{\infty}$. 
%A standard calculation using (\ref{forserre}) gives 
Using (\ref{forserre}), we get
\[
 \int t \, d\nu_{J}(t)= \frac{1}{5} \sqrt{ab}+ \frac{4}{5} \cdot \frac{a+b}{2} \approx 1.898.
\]
\begin{corollary}
%Therefore, in fact, 
There exists a sequence $\{\alpha_n\}$ of totally positive algebraic \textbf{\textit{units}} such that
\[
\liminf_{n \to \infty} {\rm{Trace}} (\alpha_n)=\liminf_{n \to \infty} \sum_{x \in {\rm {Conj}}(\alpha_n)} \frac{ x }{\deg \alpha_n} \approx 1.898.
\]
\end{corollary}
\noindent It appears that the minimal value of $\liminf_{n \to \infty} {\rm{Trace}} (\alpha_n)$ in the Schur--Siegel--Smyth 
trace problem might be attained for a sequence of algebraic units.

\begin{question} 
What is the minimal value of $\liminf_{n \to \infty} {\rm{Trace}} (\alpha_n)$ over sequences of algebraic units?
Moreover, does there exist a compact set $K \subset [0, \infty)$ with $\Cancap(K)=1$ for which this minimal value is attained?
\end{question}

To establish Theorem~\ref{T:units_equi}, we examine the asymptotic distribution of zeros of certain polynomial 
mappings in two variables.
Let $\{p_n\}$ be a sequence of polynomials that satisfies the hypotheses of Theorem~\ref{T:units_equi}
for some compact set $K$ in $\C$.
We consider the polynomial mappings $F_n : \mathbb{C}^2 \to \mathbb{C}^2$ defined using
polynomials $p_n$ as follows:
\[
F_n(z_1, z_2) = \Big(z_2^{\deg p_n}\,p_n \Big(\frac{z_1}{z_2} \Big), \ {z_1}^{m_n} z_2^{\deg p_n- m_n}-1\Big),
\]
where $m_n \in \N$ depends on the parameters $s_1$ and $s_2$ as discussed earlier corresponding to the set $K$.
Then we lift the set $K$ to $\C^2$ and investigate the equidistribution of the zeros of $F_n$ on this lift of $K$.
Indeed, we derive an equidistribution result for more general polynomial mappings and sets in $\C^2$,
which we now proceed to explain.
\smallskip

Let $\nice$ be a compact subset of $\C^2$. 
We denote by $\homcap(\nice)$ the homogeneous capacity of $\nice$
introduced by DeMarco \cite{DeM:dorm03}.
For sets with $\homcap(\nice) >0$,
let $\mu_{\nice}$ denote the Monge-Amp\`ere measure of $\nice$.
It turns out that if $\nice$ is circled\,---\,i.e., if $(z_1, z_2) \in \nice$
implies $(e^{it} z_1, e^{it} z_2) \in \nice$ for any real $t$\,---\,then $\mu_{\nice}$ is the unique $S^1$-invariant measure
that minimizes homogeneous energy for $\nice$
(see Sections~\ref{S:homopotential} and~\ref{S:nonregular} for more details).
This fact will play a crucial role in the proof of our next main result.
Consider a polynomial mapping $F: \C^2 \to \C^2$ given by
\begin{equation}\label{E:homconst}
F(z_1,z_2)= (F_1(z_1,z_2)+a_1, F_2(z_1,z_2)+a_2),
\end{equation}
where $F_1$ and $F_2$ are homogeneous polynomials of the same degree $d$.
We call $d$ the degree of $F$ and denote it by $\deg F$.
Additionally, such a mapping $F$ has $d^2$ zeros, counted with multiplicity.
We define the {\em height of $F$ relative to $\nice$} as
\[
\hgttwo_{\nice}(F):=\frac{1}{(\deg F)^2} \Big ( \log \vert \res (F) \vert 
+ \sum_{\zero \in F^{-1}(0,0)} \robin_{\nice}^+ (\zero) \Big ),
\]
where $\res (F)$ is the resultant of $F_1$ and $F_2$, as defined in Section~\ref{SS:DiscRes}
and $\robin_{\nice}^+ (\zero)$ is an analogue of the Green's function in this setting.
\smallskip

Observe that for $F$ as in \eqref{E:homconst}, if $\zero \in \C^2$ is a zero of $F$, then
$c\,\zero$ is also a zero of $F$ for any $d^{th}$ root of unity $c$.
We say that $F$ is {\em generic} if there exist $d$ zeros of $F$,
say $\zero_1, \dots , \zero_d \in \C^2$, such that 
$\zero_i$ is {\bf not} a scalar multiple of $\zero_k$ for $i \neq k$.
This condition is analogous to the condition that a univariate polynomial
$p(z)$ has simple zeros. For such mappings, we prove the following equidistribution result:

\begin{theorem}\label{T:homo_equi}
Let $\nice$ be a compact, circled set in $\C^2$ with $\homcap(\nice) =1$.
Let $\{F_n\}$ be a sequence of generic mappings with integer coefficients of the form
\eqref{E:homconst} such that
$\hgttwo_{\nice}(F_n) \to 0$ and $\deg F_n \to \infty$ as $n \to \infty$. Then
%the zeros of $F_n$ are equidistributed with respect to
\[
\mu_n:=\frac{1}{(\deg F_n)^2} \sum_{F_n(\zero)=(0,0)} \delta_{\zero} \to \mu_{\nice} \quad \text{as } n \to \infty
\]
in the weak* topology. 
\end{theorem}

Theorem~\ref{T:homo_equi} appears to be a new result even for compact sets like polydisks and balls. 
The homogeneous capacity of the polydisk $\{|z_1| \leq r_1\} \times \{|z_2| \leq r_2\}$ is $r_1 r_2$, and $\mu_{\nice}$ 
is the product of normalized arc-length measures on $\{|z_1| = r_1\}$ and $\{|z_2| = r_2\}$. For 
the Euclidean ball of radius $r$ centered at the origin, the homogeneous capacity is $r^2 e^{-1/2}$, and $\mu_{\nice}$ is 
normalized surface area measure on its boundary. The assumption in Theorem~\ref{T:homo_equi} that the mappings $F_n$ are generic is essential. 
To see this, consider the example $F_n(z_1, z_2) = (z_1^n, z_2^n - 1)$ with 
$\nice = \{|z_1| \leq 1\} \times \{|z_2| \leq 1\}$.
In this case, $\hgttwo_{\nice}(F_n) \to 0$ and $\deg F_n \to \infty$ as $n \to \infty$, but $\mu_n$ does not converge to $\mu_{\nice}$.
\smallskip

We briefly mention the idea of the proof of Theorem~\ref{T:homo_equi}.
For $\mathbf z = (z_1, z_2)$ and $\mathbf w = (w_1, w_2)$ in $\C^2$, put 
$\mathbf z \wedge \mathbf w := z_1w_2 - z_2 w_1.$ Observe that
$\mathbf z \wedge \mathbf w \neq 0$ if and only if $\mathbf z$
and $\mathbf w$ are linearly independent.
%Let $F$ be a polynomial mapping with integer coefficients of the form \eqref{E:homconst}.
If $F$ of degree $d$ has integer coefficients, and
if $\zero_1, \dots , \zero_d \in \C^2$
are any $d$ zeros of $F$ such that $\zero_i \wedge \zero_k \neq 0$ for $i \neq k$,
then we show that
\[
\prod_{i \neq k} \vert \zero_i \wedge \zero_k \vert \geq \vert \res (F) \vert^{\frac{2-2d}{d}}.
\]
This is a key ingredient in the proof of Theorem~\ref{T:homo_equi}.
%Now, let $\{F_n\}$ be a sequence of mappings as in Theorem~\ref{T:homo_equi}. 
Assume that a subsequence of $\{\mu_n\}$ in Theorem~\ref{T:homo_equi} converges to $\mu$.
Owing to the above inequality,
we show that the homogeneous energy of $\mu$ is less than or equal to $0$. However, since
$\homcap(\nice)=1$, the homogeneous energy of $\mu_\nice$ is $0$ from which it follows that $\mu= \mu_\nice$.

\medskip

{\noindent{\bf Outline of the paper:}} For simplicity of exposition, 
we prove Theorem~\ref{T:units_equi} under the assumption that $K$ is regular, meaning the Green's functions 
are continuous, in Section~\ref{S:units_equi}. To do so, it suffices to prove Theorem~\ref{T:homo_equi} for the 
case where $\nice$ is compact, circled and pseudoconvex, which we establish in Section~\ref{S:homo_equi}. 
The necessary potential-theoretic and
pluripotential-theoretic notions are discussed in Sections~\ref{S:prelim} and~\ref{S:homopotential}. 
Finally, in a brief appendix, Section~\ref{S:nonregular}, we discuss the minor modifications necessary 
to prove Theorems~\ref{T:units_equi}~and~\ref{T:homo_equi} in full generality.

\section{Potential theory in the complex plane}\label{S:prelim}
To define the capacity of a subset of $\C \setminus \{0\}$ 
with respect to 0 and $\infty$, we need a notion of the value of a matrix.
\subsection{The value of a matrix}\label{S:value}
%To define the capacity of a subset of $\C \setminus \{0\}$ 
%relative to 0 and $\infty$, we need a notion of the value of a matrix. 
Let
\[
\Gamma = \begin{bmatrix}
\gamma_{11} & \gamma_{12} \\
\gamma_{21} & \gamma_{22}
\end{bmatrix}
\]
be a symmetric $2 \times 2$ matrix with real entries, where 
the off-diagonal elements are non-negative.
Define $ \mathscr{P} $ as the set of $2$-dimensional vectors 
$ \mathbf{s} = (s_1, s_2)$ such that $s_1+s_2= 1$ and 
$ s_i \geq 0 $ for $i=1, 2$.
The {\emph{value}} of $ \Gamma$ is defined as
\begin{equation}\label{eq:def}
\text{val}(\Gamma) := \max_{\mathbf{s} \in \mathscr{P}} [\min \{\gamma_{11}s_1 + \gamma_{12}s_2, \gamma_{21}s_1 + \gamma_{22}s_2\}].
\end{equation}
By the {fundamental theorem of game theory}, we have
\begin{equation}\label{eq:fun_def}
\text{val}(\Gamma) = \min_{\mathbf{s} \in \mathscr{P}}  [\max \{\gamma_{11}s_1 + \gamma_{12}s_2, \gamma_{21}s_1 + \gamma_{22}s_2\}].
\end{equation}
%It follows from the compactness of $\mathscr{P}$ that there always exist a vector 
%$\mathbf{s} \in \mathscr{P}$ giving the optimal value.
If there exists a vector $\mathbf{s} \in \mathscr{P}$ such that
$\gamma_{11}s_1 + \gamma_{12}s_2 = \gamma_{21}s_1 + \gamma_{22}s_2$,
say equal to $v$, then by \eqref{eq:def}, we have
$\text{val}(\Gamma) \geq v$.
On the other hand, by \eqref{eq:fun_def}, we also have $\text{val}(\Gamma) \leq v$. 
Therefore, it follows that $\text{val}(\Gamma) = v$.
\smallskip

It follows from Frobenius's theorem that if 
$\text{val}(\Gamma) \neq 0$ or
$\gamma_{12}= \gamma_{21} >0$,
there exists a vector
$\mathbf{s} \in \mathscr{P}$ with $s_1, s_2 > 0$ such that 
$\gamma_{11}s_1 + \gamma_{12}s_2 = \gamma_{21}s_1 + \gamma_{22}s_2$
(see \cite[Theorem 1.4.1]{Cantor:edtdsa80}, \cite[Theorem 5.1.6]{rumely:capacitytheory89}).
If $\text{val}(\Gamma)= 0$ and $\gamma_{12}= \gamma_{21} =0$, we have the following cases:
\begin{enumerate}
\item Both $\gamma_{11}$ and $\gamma_{22}$ are zero:
For any choice of probability vector $\mathbf{s} \in \mathscr{P}$, we have 
$\gamma_{11}s_1 + \gamma_{12}s_2 = \gamma_{21}s_1 + \gamma_{22}s_2=0$.
\item One of $\gamma_{11}$ and $\gamma_{22}$ is zero and the other is non-zero:
Without loss of generality, assume $\gamma_{11}=0$ and $\gamma_{22}\neq 0$.
Then, for the vector $s_1=1$ and $s_2=0$, we have 
$\gamma_{11}s_1 + \gamma_{12}s_2 = \gamma_{21}s_1 + \gamma_{22}s_2=0$.
\item One of $\gamma_{11}$ and $\gamma_{22}$ is positive and the other is negative:
In this case, no vector $\mathbf{s} \in \mathscr{P}$ exists that satisfies 
$\gamma_{11}s_1 + \gamma_{12}s_2 = \gamma_{21}s_1 + \gamma_{22}s_2$.
For instance, consider the matrix
\[
\Gamma = \begin{bmatrix}
1 & 0 \\
0 & -1
\end{bmatrix}.
\]
\end{enumerate}
Therefore, there exists a vector $\mathbf{s} \in \mathscr{P}$ such that 
$\gamma_{11}s_1 + \gamma_{12}s_2 = \gamma_{21}s_1 + \gamma_{22}s_2$,
except in the case (3) described above.
%{\red 1. Confused about if we should talk about $\text{val}(\Gamma) \neq 0$ case?}

\subsection{The capacity with respect to 0 and $\infty$}\label{S:Cantor}
We begin by recalling some basic notions in potential theory. We refer
the reader to \cite{ransford:ptCp95} or \cite{SaffTotik:lpwef97} for more detailed discussion.
Given a probability measure $\nu$ on $\C$ with compact support,
$$U^{\nu}(z):= \int_{\C} \log \frac{1}{|z-w|}\,d\nu(w)$$
is the  {\em logarithmic potential} of $\nu$, and
$$I(\nu):=\int_{\C} \int_{\C} \log \frac{1}{|z-w|}\, d\nu(w)d\nu(z)$$
is the {\em logarithmic energy} of $\nu$.
For $K$ a compact subset of $\C$, let $\mathcal M(K)$ denote 
the set of Borel probability measures on $K$. The {\em logarithmic capacity} of $K$ is
$$\capa (K):=\exp (-\inf_{\nu \in \mathcal M(K)} I(\nu)).$$
%where the infimum is taken over all ${\nu \in \mathcal M(K)}$.
Recall that $E \subset \C$ is {\it polar} if $E\subset \{u= -\infty\}$ for some $u\not \equiv -\infty$ 
subharmonic on a neighborhood of $E$. 
If ${\rm cap} (K)$ is positive, which occurs precisely 
when $K$ is not polar, 
there is a unique $\nu_K^{\infty} \in \mathcal M(K)$ that minimizes the energy.
The measure $\nu_K^{\infty}$ is called the {\em equilibrium measure} for $K$.
The measure $\nu_K^{\infty}$ puts no mass on polar sets. 
The function 
\[
\gr_K(z,\infty):=I(\nu_K^{\infty})-U^{\nu_K^{\infty}}(z),
\]
which is subharmonic on $\C$, is the {\em Green's function of $K$ with pole at $\infty$}. 
The potential function 
$U^{\nu_K^{\infty}}=I(\nu_K^{\infty})$ quasi-everywhere (q.e.) on $K$, 
i.e., on $K\setminus E$ where $E$ is a polar set (Frostman's theorem). Hence $\gr_K(z,\infty)=0$ q.e. on $K$. 
We say $K$ is {\it regular} if $\gr_K(z,\infty)$ is continuous; in this case, $E=\emptyset$.  
Since $-U^{\nu_K^{\infty}}(z)= \log |z| + o(1)$ as $z \to \infty$, we have
\begin{equation}\label{E:greeninfty}
I(\nu_K^{\infty})= \lim_{z \to \infty} [ \gr_K(z, \infty) - \log |z|].
\end{equation}
Let $K$ be a non-polar compact subset of $\sph \setminus \{0, \infty\}$. 
%Following \cite{Cantor:edtdsa80}, given a component $D$ of $\sph \setminus K$, we let 
%$\tilde \gr_D(z,w)$ be the unique function defined for $z,w\in D$ with $z\not = w$ which is positive 
%and harmonic in each variable;  satisfies $\tilde \gr_D(z,w)=\tilde \gr_D(w,z)$; 
%$\lim_{z\to z_0} \tilde \gr_D(z,w)=0$ q.e. $z_0\in \partial D$; and for each 
%$w\not =\infty$ in $D$, $z\to \tilde \gr_D(z,w)+\log|z-w|$ has a harmonic extension to a neighborhood of 
%$w$ while $z \to \tilde \gr_D(z,\infty)-\log|z|$ has a finite limit as $z\to \infty$. 
%Then we define the generalized Green's function $\gr_K$ of $K$ for all $z,w\in \C$ by the formula 
%$$\gr_K(z,w)= \tilde \gr_D(z,w) \ \hbox{if} \ z,w \ \hbox{are in the same component} \ D \ \hbox{of} \ \sph \setminus K$$
%and otherwise $\gr_K(z,w)=0$. 
We define $\gr_K(z, 0)$, the {\em Green's function of $K$ with pole at $0$}, by the formula 
$$\gr_K(z, 0):=\gr_{1/K}(1/z,\infty)$$
where $1/K:=\{z\in \sph: 1/z\in K\}$. We call the push-forward of $\nu_{1/K}^{\infty}$ 
under the map $z\to 1/z$ the {\em harmonic 
measure of $K$ with respect to $0$} and denote this measure by $\nu_K^0$.
It follows that
\begin{equation}\label{E:greenzero}
- \int \int \log \left \vert \frac{z-w}{zw} \right \vert d\nu_K^0(z) d\nu_K^{0}(w) = \lim_{z \to 0} [\gr_K(z, 0) + \log |z|].
\end{equation}

We are now ready to define the capacity with respect to 0 and $\infty$. Consider the matrix
\[
\Gamma (K):=
\begin{bmatrix}
\displaystyle \lim_{z \to 0} [\gr_K(z, 0) + \log |z|] & \gr_K(0, \infty) \\
\gr_K(\infty, 0) & \displaystyle \lim_{z \to \infty} [ \gr_K(z, \infty) - \log |z|]
\end{bmatrix}.
\]
Define the {\em Robin's constant of $K$ with respect to 0 and $\infty$} to be 
$\robcons(K):=\textnormal{val}(\Gamma (K))$.
The {\em capacity of $K$ with respect to 0 and $\infty$} is 
\[
\Cancap(K):=e^{-\robcons(K)}.
\]
\begin{remark}
In fact, Cantor \cite{Cantor:edtdsa80} introduced the capacity of an adelic set with respect to a finite subset of 
the Riemann sphere $\sph$. We use here only the Archimedean part of his definition with respect to $0$ and $\infty$.
\end{remark}

\section{Homogeneous potential theory}\label{S:homopotential}

We now recall the homogeneous potential theory introduced by DeMarco \cite{DeM:dorm03}.
The definitions are analogous to the
logarithmic potential, energy and capacity in the complex plane.
For $\mathbf z = (z_1, z_2)$ and $\mathbf w = (w_1, w_2)$ in $\C^2$, put 
$$\mathbf z \wedge \mathbf w := z_1w_2 - z_2 w_1.$$
Let $\Vert \cdot \Vert$ denote the Euclidean norm on $\C^2$.
It follows that $|\mathbf z \wedge \mathbf w| \leq 2 \Vert \mathbf z \Vert  \Vert \mathbf w \Vert$.
A function $f: \C^2 \to \R \cup \{-\infty\}$ is called {\em logarithmically homogeneous} if for
$\lambda \in \mathbb{C}\setminus\{0\}$, we have $f(\lambda \mathbf z) = f(\mathbf z)+ \log |\lambda|.$
For a probability measure $\mu$ on $\C^2$ with compact support,
the {\em homogeneous potential} of $\mu$ is given by
$$
\U^\mu(\mathbf z) := \int_{\C^2} \log \frac{1}{\vert \mathbf z \wedge \mathbf w \vert} \, d\mu(\mathbf w).
$$
Observe that $-\U^\mu$ is plurisubharmonic and logarithmically homogeneous.
The {\em homogeneous energy} of $\mu$ is given by
\[
\I(\mu) :=  \int_{\C^2} \int_{\C^2} \log \frac{1}{\vert \mathbf z \wedge \mathbf w \vert} \, d\mu(\mathbf w) d\mu(\mathbf z).
\]
%We use the same notation for the homogeneous potential and energy of a measure on $\C^2$ as for the logarithmic potential 
%and energy of a measure on $\C$. However, this should not cause confusion, as the context will clarify which 
%concepts we are referring to. 
The {\em homogeneous capacity} of a compact set $\nice \subset \C^2$ is defined as
$$\homcap(\nice):=\exp (-\inf_{\mu \in \mathcal M(\nice)} \I(\mu)).$$
%where the infimum is taken over $\mathcal M(\nice)$, the probability measures supported in $\nice$.
%For compact and circled $\nice$, we have $\nice$ is pluripolar if and only if $\homcap (\nice)=0$.

\smallskip

Let 
$
L(\mathbb{C}^2) := \{ u \text{ plurisubharmonic in } \mathbb{C}^2 : u(\mathbf z) \leq \log^+ \Vert \mathbf z \Vert + c \},
$
where the constant $c$ may depend on $u$.
For $\nice \subset \mathbb{C}^2$ compact, define 
\[
\ext_\nice(\mathbf z) := \sup \{ u(\mathbf z) : u \in L(\mathbb{C}^2), u \leq 0 \text{ on } \nice \}.
\]
The upper semicontinuous regularization $
\ext_\nice^*(\mathbf z) := \limsup_{\mathbf w \to\mathbf z} \ext_{\nice}(\mathbf w)$ is called the 
{\em pluricomplex Green's function} of $\nice$. Either $\ext_\nice^* \equiv +\infty$, which occurs precisely when 
$\nice$ is pluripolar, i.e., $\nice \subseteq \{ u = -\infty \}$ for some plurisubharmonic function $u \not\equiv -\infty$
in a neighborhood of $\nice$, or $\ext_\nice^* \in L(\mathbb{C}^2)$. 
In the latter case, the {\em Monge--Amp\`ere measure}
\begin{equation}\label{mameas} \mu_{\nice}:= dd^c \ext_\nice^*  \wedge dd^c \ext_\nice^* \end{equation} 
is a well defined probability measure supported in $\nice$. 
Here $d=\partial +\bar \partial$ and $d^c =i(\bar \partial +\partial)/2\pi$ so 
that $dd^c =i\partial \bar \partial/\pi$. We remark that in $\C$, replacing ``plurisubharmonic'' by ``subharmonic'' 
we recover the Green's function 
$\gr_{\nice}(\cdot,\infty)=  \ext_\nice^*(\cdot)$ of $\nice \subset \C$ with pole at $\infty$; 
in this case $\mu_{\nice}$ is the Laplacian of $\gr_{\nice}(\cdot,\infty)$. Finally, we say that $\nice$ is regular if $\ext_\nice $ 
is continuous; equivalently, $\ext_\nice = \ext_\nice^*$ (see \cite{Klimek91} for details).
\smallskip

The {\em Robin function} of a nonpluripolar set $\nice$ is defined as
\[
\robin_\nice(\mathbf z) := \limsup_{|\lambda| \to \infty} \left[\ext_\Sigma^*(\lambda \mathbf z) - \log |\lambda| \right].
\]
Observe that $\robin_\nice$ is logarithmically homogeneous. The definitions of $\ext_\nice$ and $\robin_\nice$ make sense for general bounded 
Borel sets $\nice$, and it is known that for $\nice$ nonpluripolar the Robin function $\robin_\nice(\mathbf z)$ 
is plurisubharmonic in $\mathbb{C}^2$. If $\nice$ is compact, circled, and nonpluripolar, then
\[
\ext_\Sigma^*( \mathbf z) = \robin^+_{\nice}(\mathbf z) :=\max \{0, \robin_\nice(\mathbf z) \}
\]
and $\mu_{\nice}$ satisfies ${\rm supp} (\mu_{\nice}) \subseteq \{\robin_\nice=0\}$ (see \cite{BLL:hilbert}). 
If $\nice$ is compact, circled and pseudoconvex as defined in \cite{DeM:dorm03}, i.e., $\nice$ is the closure of an $S^1$-invariant, bounded, pseudoconvex domain
in $\C^2$ containing the origin,  then
$\robin_\nice$ has a nice geometrical interpretation:
\[
\robin_\nice (\mathbf z)= \inf \{-\log \lambda: \lambda \mathbf z \in \nice\}.
\] 
In this case $\robin_{\nice}$ is a continuous, plurisubharmonic function
on $\C^2$ that scales logarithmically, $\robin_{\nice}^{-1}(-\infty)=\{(0, 0)\}$, and 
$$\ext_\Sigma( \mathbf z) = \robin^+_{\nice}(\mathbf z) \ \hbox{and} \ \mu_{\nice}= dd^c \robin_{\nice}^+ \wedge dd^c \robin_{\nice}^+.$$  
Conversely, if
$\nice = \{\mathbf z : f(\mathbf z ) \leq 0\}$ for a continuous, plurisubharmonic function
$f : \C^2 \to \R \cup \{-\infty\}$ which is logarithmically homogeneous and 
$f^{-1}(-\infty)=\{(0, 0)\}$, then $\nice$ is compact, circled and pseudoconvex
and $\robin_{\nice} \equiv f$. 
\smallskip

Let $\pi: \C^2 \setminus \{(0,0)\} \to \sph$
given by $\pi (z_1, z_2)= z_1/z_2$ if $z_2 \neq 0$ and $\pi (z_1, 0)= \infty$. The following two results were proved by DeMarco in the compact, circled and pseudoconvex (i.e., regular) case; they remain valid without the regularity assumption; i.e., for compact, circled and nonpluripolar sets. We indicate the simple modifications to DeMarco's proofs in an appendix.

\begin{result}[\cite{DeM:dorm03}]\label{pushf}
Let $\nice \subset \C^2$ be compact, circled and nonpluripolar and let $\nu$ be the unique probability measure on $\sph$
such that $dd^c (\robin_{\nice}) = \pi^* \nu$. Then
$\pi_* (\mu_{\nice})= \nu$.
\end{result}

For every compact set, there always exists a homogeneous energy minimizing measure on $\nice$.
In general, such a measure need not be unique. However, for compact, circled and nonpluripolar sets we have uniqueness 
if we restrict attention to $S^1$-invariant measures.

\begin{result}[\cite{DeM:dorm03}]\label{R:unique}
For any compact, circled and nonpluripolar $\nice \subset \C^2$, the measure $\mu_\nice$ is the unique, 
$S^1$-invariant homogeneous energy minimizing measure on $\nice$.
%The homogeneous potential function of $\mu_\nice$ is constant on the boundary of
%$\nice$ and satisfies
%\[
%-\U^{\mu_\nice} = \robin_\nice + \log(\homcap(\nice)).
%\]
\end{result}

\noindent In particular, the homogeneous capacity of such sets is always positive.

\subsection{Discriminant and Resultant}\label{SS:DiscRes}
Consider a polynomial mapping $F: \C^2 \to \C^2$ of the form
$
F(z_1,z_2)= (F_1(z_1,z_2) + a_1, F_2(z_1,z_2)+ a_2),
$
where $F_1, F_2 \in \C[z_1, z_2]$ are homogeneous polynomials of degree $d$.
We define the resultant of $F$ as the resultant of $F_1(z_1, 1)$ and $F_2(z_1, 1)$
considered as polynomials of degree $d$
(their degree in $z_1$ may be lower than their total degree):
\begin{equation}\label{resdef}\res (F)=\trueres_{d,d}(F_1(z_1, 1), F_2(z_1, 1)).\end{equation}
For example, if a polynomial $p$ is of degree $d$,
and $l$ denotes the degree of $q$ with $l < d$, then we have
(see \cite[Chapter 12]{GKZ:drmd08})
\begin{equation}\label{resdef2}
\trueres_{d,d}(p, q) = (-1)^{d(d-l)} {\rm lead} (p)^d{\rm lead} (q)^d \prod_{i,j} (x_i - y_j),
\end{equation}
where ${\rm lead} (p)$ denotes the leading coefficient of $p$; $x_1, \dots, x_d$ are the roots of $p$; and $y_1, \dots, y_l$ are the roots of $q$.
If the polynomials $p$ and $q$ have integer coefficients then $\trueres_{d,d}(p, q) \in \Z$.
Note that $\res (F)$ vanishes if and only if $F_1$ and $F_2$ have a
common factor.
If $F_1$ and $F_2$ are factored as
$F_1(z_1, z_2)= \prod_{i=1}^d (z_1, z_2) \wedge \mathbf a_i$ and
$F_2(z_1, z_2)= \prod_{j=1}^d (z_1, z_2) \wedge \mathbf b_j$
for some $\mathbf a_i$ and $\mathbf b_j$ in $\C^2$, then
$$\res (F)=\prod_{i,j}\mathbf a_i \wedge\mathbf b_j.$$
%where $(a_1, a_2) \wedge (b_1, b_2)= a_1b_2-b_1a_2$.
If $\Phi: \C^2 \to \C^2$ is a linear mapping, then we have (see \cite[Proposition 6.1]{DeM:dorm03})
\begin{equation}\label{E:composition}
\res (\Phi \circ F)= \res (\Phi)^{\deg F} \res(F).
\end{equation}
Lastly, if $p$ is a polynomial of degree $d$ with roots $x_1, \dots, x_d$ then the discriminant of $p$ is
given by
\[
\discr(p)= (-1)^{\frac{d(d-1)}{2}}  {\rm lead} (p)^{2d-2} \prod_{i \neq j} (x_i-x_j).
\]
If the polynomial $p$ has integer coefficients then $\discr(p) \in \Z$.

\section{Polynomial mappings with integer coefficients in $\C^2$}\label{S:homo_equi}

In this section, we present an equidistribution result for the zeros of certain
polynomial mappings with integer coefficients.
%Specifically, we explore the distribution of the zeros of these mappings and demonstrate that, under 
%certain conditions, these zeros are equidistributed with respect to the homogeneous equilibrium measure 
%of the set with homogeneous capacity 1. 
To establish this result, we first derive a relationship between specific zeros of polynomial mappings 
$F: \C^2 \to \C^2$ of the form 
\begin{equation}\label{eq:homo_poly}
F(z_1,z_2)= (F_1(z_1,z_2) + a_1, F_2(z_1,z_2)+ a_2),
\end{equation}
where $F_1, F_2 \in \C[z_1, z_2]$ are homogeneous polynomials of degree $d \geq 2$.
Note that if $\zero \in \C^2$ is a zero of $F$, then $c\,\zero$ is also a zero of $F$, where $c$
is a $d$-th root of unity. Consequently, the product 
$\prod_{i \neq k} \vert \zero_i \wedge \zero_k \vert$ taken over all $d^2$ zeros of $F$ is always zero. 
However, in the case of certain polynomials, when we consider only a specific subset of the zeros, 
we obtain a relationship involving $\res (F)$ (recall (\ref{resdef}) for this definition).

\begin{lemma}\label{L:complex_coeff}
Let $F: \C^2 \to \C^2$ be a polynomial mapping of the form 
%\eqref{eq:homo_poly} with $a_1=0$ and $a_2=-1$.
$$F(z_1,z_2)= (F_1(z_1,z_2), F_2(z_1,z_2)-1),$$
where $F_1, F_2 \in \C[z_1, z_2]$ are homogeneous polynomials of degree $d \geq 2$.
Let $\zero_1, \dots , \zero_d \in \C^2$
be any $d$ zeros of $F$ such that $\zero_i \wedge \zero_k \neq 0$ for $i \neq k$.
Then the following identity holds:
%\[
%\vert \res (F) \vert^{\frac{2d-2}{d}}\prod_{i \neq k} \vert \zero_i \wedge \zero_k \vert =\vert\discr(f)\vert,
%\]
\[
 \prod_{i \neq k} \vert \zero_i \wedge \zero_k \vert =\vert \res (F) \vert^{\frac{2-2d}{d}} \vert\discr(p)\vert,
\]
where $\discr(p)$ is the discriminant of the univariate polynomial $p(x) := F_1(x,1)$.
\end{lemma}

\begin{proof}
Since both $F_1$ and $F_2$ are homogeneous polynomials in two variables,
we can factor $F_1$ and $F_2$ as
\[
F_1(z_1, z_2)= \prod_{i=1}^d (z_1, z_2) \wedge \mathbf a_i \quad {\text{and}} \quad 
F_2(z_1, z_2)= \prod_{j=1}^d (z_1, z_2) \wedge \mathbf b_j
\]
for some $\mathbf a_i$ and $\mathbf b_j$ in $\C^2$.
Observe that $F(\mathbf a_i)=(0, (\prod_j \mathbf a_i \wedge \mathbf b_j)-1)$. 
%Note that the zero set of $F$ is same as $F_h^{-1}(0,1)$.
Therefore, all $d^2$ zeros of $F$
are given by $$\frac{\mathbf a_i}{(\prod_j \mathbf a_i \wedge \mathbf b_j)^{1/d}}$$
for all $1 \leq i \leq d$ and all $d$-th roots of $\prod_j \mathbf a_i \wedge \mathbf b_j$. 
%Let $\zero_1, \dots , \zero_d$
%be any $d$ preimages of $(0,1)$ with $\zero_i \wedge \zero_k \neq 0$ for $i \neq k$.
Now, consider
\begin{align}
\prod_{i \neq k} \vert \zero_i \wedge \zero_k \vert 
=& \prod_{i \neq k}\Big\vert \frac{\mathbf a_i}{(\prod_j \mathbf a_i \wedge \mathbf b_j)^{1/d}} \wedge \frac{\mathbf a_k}{(\prod_j \mathbf a_k \wedge \mathbf b_j)^{1/d}}\Big\vert \notag \\
%=& \prod_{i \neq k}\Big\vert \frac{\mathbf a_i}{\vert(\prod_j \mathbf a_i \wedge \mathbf b_j)^{1/d}\vert} \wedge \frac{\mathbf a_k}{\vert(\prod_j \mathbf a_k \wedge \mathbf b_j)^{1/d}\vert}\Big\vert \notag \\
=& \prod_{i} \frac{1}{\vert (\prod_j \mathbf a_i \wedge \mathbf b_j)^{1/d}\vert^{d-1}} 
\prod_{k} \frac{1}{\vert(\prod_j \mathbf a_k \wedge \mathbf b_j)^{1/d}\vert^{d-1}} 
\prod_{i \neq k} \vert \mathbf a_i \wedge \mathbf a_k \vert \notag \\
=& \frac{1}{\vert \res (F) \vert^{(2d-2)/d}} \prod_{i \neq k} \vert \mathbf a_i \wedge \mathbf a_k \vert. \label{E:roots_wedge}
\end{align}
We can express $F_1$ as
$F_1(z_1, z_2)= z_2^d \cdot p (z_1/z_2)$,
where $p(x)$ is the univariate polynomial $F_1(x,1)$.
If $p(x)= c \prod_i (x-r_i)$,
then
$F_1(z_1, z_2)= c \prod_i (z_1-r_i z_2)=c\prod_i (z_1, z_2) \wedge (r_i, 1)$. 
Thus, by relabeling if necessary, for each $i$, we have
$\mathbf a_i= (c_i r_i, c_i)$
%$$\mathbf a_1= (c_1r_1, c_1), \  \mathbf a_2= (c_2 r_2, c_2), \dots, \mathbf a_d= (c_d r_d, c_d).$$
for some constants $c_i \in \C$ satisfying $\prod_i c_i=c$.
This gives us
\begin{equation}\label{E:coeff_wedge}
\prod_{i \neq k} \mathbf a_i \wedge \mathbf a_k = \prod_i c_i^{2d-2} \prod_{i \neq k} (r_i - r_k)
=c^{2d-2} \prod_{i \neq k} (r_i - r_k) = (-1)^{\frac{d(d-1)}{2}}\discr (p).
\end{equation}
The proof follows from \eqref{E:roots_wedge} and \eqref{E:coeff_wedge}.
\end{proof}

We now apply Lemma~\ref{L:complex_coeff} to obtain a lower bound on the pairwise wedge products of
zeros of polynomial mappings of the form \eqref{eq:homo_poly} with integer coefficients.
\begin{lemma}\label{L:funda_integer}
Let $F: \C^2 \to \C^2$ be a polynomial mapping of the form 
%$$F(z_1,z_2)= (F_1(z_1,z_2)+m, F_2(z_1,z_2)+n),$$
\eqref{eq:homo_poly} with integer coefficients.
Let $\zero_1, \dots , \zero_d \in \C^2$
be any $d$ zeros of $F$ such that $\zero_i \wedge \zero_k \neq 0$ for $i \neq k$.
Then the following inequality holds:
%\[
%\vert \res (F) \vert^{\frac{2d-2}{d}}\prod_{i \neq k} \vert \zero_i \wedge \zero_k \vert \geq 1.
%\]
\[
\prod_{i \neq k} \vert \zero_i \wedge \zero_k \vert \geq \vert \res (F) \vert^{\frac{2-2d}{d}}.
\]
\end{lemma}

\begin{proof}
We begin by noting that at least one of $a_1$ and $a_2$ must be non-zero.
If both $a_1$ and $a_2$ were zero, there would be no $\zero_i$'s satisfying the hypothesis.
Next, assume that $a_1=0$ and $a_2=-1$. 
Under the assumption on the $\zero_i$'s, 
by Lemma~\ref{L:complex_coeff}, $\discr(p) \neq 0$ for
the polynomial $p(x):=F_1(x,1)$.
Since $F$ has integer coefficients, we know that $ \vert\discr(p)\vert \geq 1$. 
Thus, applying Lemma~\ref{L:complex_coeff}, we obtain the desired inequality.
\smallskip

Now, suppose that $a_1$ and $a_2$ are such that $\vert \gcd(a_1, a_2)\vert=1$.
In this case, we can reduce to the situation considered above.
By Bézout's identity,
there exist integers $x$ and $y$ such that $a_1x + a_2y = 1$.
Now consider the polynomial mapping with integer coefficients
given by
$\Phi(z_1, z_2)=(-a_2 z_1+ a_1z_2, -x z_1- y z_2)$.
%the matrix
%\[
 %  M=
  %\left[ {\begin{array}{cc}
   %-c_2 & c_1 \\
   %x & y \\
 % \end{array} } \right],
%\]
It is easy to verify that $\vert \res (\Phi) \vert =1$ and $\Phi (-a_1, -a_2)=(0,1)$. 
Observe that $\vert \res (\Phi \circ F) \vert = \vert \res (F) \vert$ from \eqref{E:composition}, and
the zeros of $F$ and $\Phi \circ F$ are the same. Moreover,
$\Phi \circ F$ is of the form \eqref{eq:homo_poly}
with $a_1=0$ and $a_2=-1$.
\smallskip

Finally, if $\vert \gcd(a_1, a_2)\vert=g>1$, then the zeros of $F$ are the zeros of
\[
\tilde F(z_1,z_2):= (F_1(z_1,z_2) + a_1/g, F_2(z_1,z_2)+ a_2/g)
\]
scaled by the positive $d$-th root of $g$. By the last case,
the required inequality holds for the zeros of $F$. We have $\res (F)= \res (\tilde F)$.
Thus the inequality also holds for 
the zeros of $\tilde F$.
\end{proof}

Consider a sequence of polynomial mappings $\{F_n\}$ of the form \eqref{eq:homo_poly}
such that $d_n:=\deg F_n \to \infty$ as $n \to \infty$. 
Let $\mu_n$ denote the probability measure uniformly distributed on the zeros of $F_n$.
It turns out that if the sequence $\{\mu_n\}$ converges, the limit measure possesses a certain invariance property.
\begin{lemma}\label{L:sub_inv}
Assume that the sequence $\{\mu_n\}$ converges to $\mu$ in the weak topology. Then $\mu$ is $S^1$-invariant.
\end{lemma}
\begin{proof}
As noted earlier, if $\zero$ is a zero of $F_n$, then $c_n \zero$ is also a zero of $F_n$, where $c_n$
is a $d_n$-th root of unity.
Therefore, $\mu_n$ is invariant under the action of the $d_n$-th roots of unity. Specifically, 
for any continuous bounded function $f$ on $\C^2$,
%and any $d_n$-th root of unity $c$, 
we have
\[
\int f(c_n \mathbf z) \, d\mu_n(\mathbf z) = \int f(\mathbf z) \, d\mu_n(\mathbf z).
\]
Now, for a fixed $\theta \in S^1$, 
there exists a sequence of $d_n$-th roots of unity $c_n$ that converges to $\theta$ as $n \to \infty$.
By applying the dominated convergence theorem and leveraging the weak convergence of $\{\mu_n\}$ to $\mu$,
we obtain
\[
\int f(\theta \mathbf z) \, d\mu(\mathbf z) = \lim_{n \to \infty} \int f(c_n \mathbf z) \, d\mu_n(\mathbf z) 
= \lim_{n \to \infty} \int f(\mathbf z) \, d\mu_n(\mathbf z) = \int f(\mathbf z) \, d\mu(\mathbf z).
\]
%\[
%\int f(\theta x) \, d\mu(x) = \int f(x) \, d\mu(x).
%\]
Since this holds for any continuous bounded function $f$ and any $\theta \in S^1$, 
we conclude that $\mu$ is  $S^1$-invariant.
\end{proof}

%We are now ready for an equidistribution result mentioned earlier.
Before we proceed with the proof of Theorem~\ref{T:homo_equi}, 
we first recall some important definitions.
%\[
%{\red \rho \to e^{\rho}}  \ \rho_{\nice}(z):= \inf \Big\{ \frac{1}{\vert c \vert} : cz \in {\nice} \text{ for } c \in \C \Big\},
%\]
Let $\Sigma$ be a compact and circled subset of $\C^2$ and let $F$ be a polynomial mapping  
of the form \eqref{eq:homo_poly}.  Recall the definition
\[
\hgttwo_{\nice}(F):=\frac{1}{(\deg F)^2} \Big ( \log \vert \res (F) \vert 
+ \sum_{\zero \in F^{-1}(0,0)} \robin_{\nice}^+ (\zero) \Big )
\]
and recall from Section~\ref{S:homopotential} that 
$\ext_\Sigma^*( \mathbf z) = \robin_{\nice}^+ (\mathbf z):=\max \{0, \robin_{\nice}(\mathbf z)\}$.
Assume that $\{F_n\}$ is a sequence of polynomial mappings with integer coefficients
such that $\hgttwo_{\nice}(F_n) \to 0$ as $n \to \infty$.
Since $\res (F_n) \in \Z$ and 
$\max \{0, \robin_{\nice}(\mathbf z)\} \geq 0$, we observe that
\[
{\frac{1}{(\deg F_n)^2}} \log \vert \res (F_n) \vert \to 0
\text{\ \ and \ }
{\frac{1}{(\deg F_n)^2}} \sum_{\zero \in F_n^{-1}(0,0)} \robin_{\nice}^+ (\zero) \to 0
\]
as $n \to \infty$.
These observations will be used in the following proof.

\begin{proof}[The proof of Theorem~\ref{T:homo_equi}]
We give the proof in the case where $\nice$ is compact, circled and pseudoconvex; in the appendix we indicate the minor 
modifications if $\nice$ is just compact, circled and nonpluripolar.
We first show that every subsequence of $\{\mu_n\}$ has a convergent subsequence.
Let $\Omega$ be a closed subset of $\C^2 \setminus \nice$. Observe that
\begin{align}
0 \leq \mu_n(\Omega) \min_{\zero \in \Omega \cap F_n^{-1}(0,0)} \robin_{\nice}^+ (\zero) 
\leq \frac{1}{(\deg F_n)^2} \sum_{\substack{\zero \in  \Omega \cap F_n^{-1}(0,0)}} \robin_{\nice}^+ (\zero)
\leq \hgttwo_{\nice}(F_n). \notag
\end{align}
Since
$
\hgttwo_{\nice}(F_n)\to 0
$
as $n \to \infty$, 
and $\robin_{\nice}^+(\mathbf z) \geq \delta=\delta(\Omega) >0$ for any $\mathbf z \in \Omega$,
we conclude that $\mu_n(\Omega) \to 0$ as $n \to \infty$.
This shows that the sequence $\{\mu_n\}$ is tight.
By Prokhorov's theorem, the tightness of $\{\mu_n\}$ implies that
every subsequence of $\{\mu_n\}$ has a convergent subsequence. 
Note that the limit of any such convergent subsequence is a probability measure whose support
is contained in $\nice$. If we can further show that any such limiting measure is equal 
to $\mu_{\Sigma}$, then the result will follow.
\smallskip

For simplicity, we will use the same notation $\{\mu_n\}$ to refer to any such convergent subsequence. 
Let $\mu$ denote the limit of $\{\mu_n\}$.
Let $h_M (\mathbf z, \mathbf w):= \min \{M, - \log {\vert \mathbf z \wedge \mathbf w \vert}\}$.
Observe that $h_M (\mathbf z, \mathbf w)$ is a continuous function in $\mathbf z$ and $\mathbf w$ on $\C^2 \times \C^2$, 
and that $h_M (\mathbf z, \mathbf w)$ increases to $-\log {\vert \mathbf z \wedge \mathbf w \vert}$ as $M \to \infty$.
Every continuous function $g(\mathbf z, \mathbf w)$ on a compact subset 
$\Omega \times \Omega$ of $\C^2 \times \C^2$ can be uniformly approximated 
by finite sums of the form
$\sum g_1(\mathbf z)g_2(\mathbf w)$, where $g_1$ and $g_2$ are continuous on $\Omega$. Thus $\mu_n \times \mu_n \to \mu \times \mu$
as $n \to \infty$ in the weak topology. 
Let $\lar:= \{\Vert \mathbf z \Vert  \leq r\}$. 
By the monotone convergence theorem and 
the definition of weak convergence of measures, for any sufficiently large $r>0$,
we have
\begin{align}
\I(\mu) 
%= - \int \int \log \vert z \wedge w \vert \, d\mu(z) \, d\mu(w)
&= \lim_{M \to \infty} \int_{\lar} \int_{\lar} h_M (\mathbf z, \mathbf w) \, d\mu(\mathbf z) \, d\mu(\mathbf w) \notag \\
&=\lim_{M \to \infty} \lim_{n \to \infty} \int_{\lar} \int_{\lar} h_M (\mathbf z, \mathbf w) \, d\mu_{n}(\mathbf z) \, d\mu_n(\mathbf w). \label{E:weak_con}
\end{align}
Let $d_n:= \deg F_n$.
Let $\zero_{n,1}, \dots, \zero_{n, d_n}$ be any $d_n$ zeros of $F_n$ such that 
$\zero_{n,i} \wedge \zero_{n,k} \neq 0$ for $1 \leq i \neq k \leq d_n$.
Assume that these zeros are ordered in
such a way that
$\Vert \zero_{n,i}\Vert \leq \Vert \zero_{n,k}\Vert$ for $1 \leq i<k \leq d_n$.
Let $j_n \in \N$ be such that $\zero_{n,i} \in \lar$ if and only if $i \leq j_n$.
As observed above, for sufficiently large $r>0$, $\mu_n(\lar) \to 1$ as $n \to \infty$, 
implying that $j_n/ d_n \to 1$ as $n \to \infty$.
Recall that any other zero of $F_n$ is a multiple of one of the zeros listed above, 
scaled by some $d_n$-th roots of unity.
Therefore, we have
\begin{align}
\lim_{n \to \infty} \int_{\lar} \int_{\lar} h_M (\mathbf z, \mathbf w) \, d\mu_{n}(\mathbf z) \, d\mu_n(\mathbf w) 
&\leq \lim_{n \to \infty} \Big( \frac{M d_n^3}{d_n^4} 
+ \frac{1}{d_n^4} \sum_{\substack{\zero \wedge \eta \neq 0 \\ \zero, \eta \in \lar \cap F_n^{-1}(0,0)}} h_M (\zero, \eta) \Big) \notag \\ 
&\leq \liminf_{n \to \infty} \frac{d_n^2}{d_n^4}\sum_{1 \leq i \neq k \leq j_n} \log \frac{1}{\vert \zero_{n,i} \wedge \zero_{n,k} \vert}. \label{E:lim_int}
\end{align}
To estimate the last term, consider
\begin{align}
\prod_{1 \leq i \neq k \leq d_n} \vert \zero_{n,i} \wedge \zero_{n,k} \vert
&=\prod_{1 \leq i \neq k \leq j_n} \vert \zero_{n,i} \wedge \zero_{n,k} \vert \notag
\prod_{\substack{1 \leq i < k\\ j_n <k \leq d_n}} \vert \zero_{n,i} \wedge \zero_{n,k} \vert^2\\
&\leq \prod_{1 \leq i \neq k \leq j_n} \vert \zero_{n,i} \wedge \zero_{n,k} \vert
\prod_{\substack{1 \leq i < k\\ j_n <k \leq d_n}} (2 \Vert \zero_{n,i} \Vert \, \Vert \zero_{n,k} \Vert)^2 \notag\\
&\leq \prod_{1 \leq i \neq k \leq j_n} \vert \zero_{n,i} \wedge \zero_{n,k} \vert
\prod_{j_n <k \leq d_n} (2 \Vert \zero_{n,k} \Vert^2)^{2(d_n-1)} \notag\\
&\leq 4^{(d_n-1)(d_n-j_n)} \prod_{1 \leq i \neq k \leq j_n} \vert \zero_{n,i} \wedge \zero_{n,k} \vert
\prod_{j_n <k \leq d_n} \Vert \zero_{n,k} \Vert^{4(d_n-1)}. \notag
\end{align}
Therefore, we conclude that
\begin{equation}\label{E:estimate_1}
\liminf_{n \to \infty} \prod_{1 \leq i \neq k \leq j_n} \vert \zero_{n,i} \wedge \zero_{n,k} \vert^{\frac{1}{d_n^2}}
\geq \frac{\liminf_{n \to \infty} \prod_{1 \leq i \neq k \leq d_n} \vert \zero_{n,i} \wedge \zero_{n,k} \vert^{\frac{1}{d_n^2}}}
{\limsup_{n \to \infty} \prod_{j_n <k \leq d_n} \Vert \zero_{n,k} \Vert^{\frac{4(d_n-1)}{d_n^2}}}.
\end{equation}
By Lemma~\ref{L:funda_integer} and the assumption that $\hgttwo_{\nice}(F_n) \to 0$ as $n \to \infty$, we have
\begin{equation}\label{for4.6a}
\liminf_{n \to \infty} \prod_{1 \leq i \neq k \leq d_n} \vert \zero_{n,i} \wedge \zero_{n,k} \vert^{\frac{1}{d_n^2}}
\geq \liminf_{n \to \infty} \left(\vert \res(F_n) \vert^{\frac{2-2d_n}{d_n}} \right)^{\frac{1}{d_n^2}} 
= 1.
\end{equation}
Since $\nice$ is circled and pseudoconvex, by definition, it follows that
\begin{equation}\label{E:infsup}
-\infty< \inf_{\Vert \mathbf z \Vert=1} \robin_{\nice}(\mathbf z) \leq \sup_{\Vert \mathbf z \Vert=1} \robin_{\nice}(\mathbf z) <\infty.
\end{equation}
Thus by the logarithmic homogenity of $\robin_{\nice}(\mathbf z)$, 
$\robin_{\nice}(\mathbf z) - \log \Vert \mathbf z \Vert=O(1)$ as $\Vert \mathbf z \Vert \to \infty$.
Consequently, there exists a constant $C$ such that for all $ \Vert \mathbf z  \Vert >r$ 
(with $r$ sufficiently large), we have
$ \Vert \mathbf z  \Vert \leq C \exp (\robin_{\nice}(\mathbf z))$. Now consider
\begin{align}
\prod_{j_n <k \leq d_n} \Vert \zero_{n,k} \Vert^{\frac{4(d_n-1)}{d_n^2}}
%= \prod_{\substack{\zero \in F_n^{-1}(0,0) \\ \zero \notin \lar}} {\Vert \zero \Vert^{\frac{1}{d_n}}}^{\frac{4(d_n-1)}{d_n^2}}
= \prod_{\substack{\zero \in F_n^{-1}(0,0) \\ \zero \notin \lar}} \Vert \zero \Vert^{\frac{4(d_n-1)}{d_n^3}}
&\leq C^{\frac{4(d_n-j_n) (d_n-1)}{d_n^2}}\prod_{\zero \in F_n^{-1}(0,0)} \exp (\robin_{\nice}^+ (\zero))^{\frac{4}{d_n^2}}. \notag
%&\leq C^{\frac{4(d_n-j_n) (d_n-1)}{d_n^2}} \prod_{\zero \in F_n^{-1}(0,0)} \max \{1, \rho_{\nice}(\zero)\}^{\frac{4}{d_n^2}}. \notag
\end{align}
Since $j_n/d_n \to 1$ and $\hgttwo_{\nice}(F_n) \to 0$ as $n \to \infty$, 
both terms on the right-hand side tend to $1$ as $n\to \infty$. Thus 
\begin{equation}\label{for4.6b}
\limsup_{n \to \infty} \prod_{j_n <k \leq d_n} \Vert \zero_{n,k} \Vert^{\frac{4(d_n-1)}{d_n^2}} \leq 1.
\end{equation}
Using the estimates (\ref{for4.6a}) and (\ref{for4.6b}) in \eqref{E:estimate_1}, we get
\begin{equation}\label{E:estimate_2}
\liminf_{n \to \infty} \prod_{1 \leq i \neq k \leq j_n} \vert \zero_{n,i} \wedge \zero_{n,k} \vert^{\frac{1}{d_n^2}}
\geq 1.
\end{equation}
Finally, by combining the estimate in \eqref{E:estimate_2} with \eqref{E:weak_con} and \eqref{E:lim_int},
we obtain $\I(\mu) \leq 0$. Since the homogeneous capacity of $\nice$ is $1$, it follows that $\I(\mu)=0$.
\smallskip

According to Lemma~\ref{L:sub_inv}, the measure $\mu$ is $S^1$-invariant. 
Therefore, by Result~\ref{R:unique}, any such measure 
is unique and must be equal to $\mu_{\nice}$. This completes the proof.
\end{proof}

\begin{remark} For $\nice=\{(z_1,z_2):|z_1|\leq 1, \ |z_2|\leq 1\}$ the unit polydisk, it is easy to construct 
examples of sequences $\{F_n\}$ satisfying the hypotheses of Theorem~\ref{T:homo_equi}, 
e.g., $F_n(z_1,z_2)=(z_1^n -1, z_2^n -1)$. It would be interesting to construct explicit examples for 
other compact, circled pseudoconvex sets $\nice$ of homogeneous capacity one. It is unclear how to 
do this for other polydisks $\nice =\{(z_1,z_2):|z_1|\leq r_1, \ |z_2|\leq r_2\}$ with $r_1r_2=1$ or for the 
Euclidean ball $\{\Vert \mathbf{z} \Vert \leq e^{1/4}\}$.
\smallskip

It would be interesting to determine some reasonable and verifiable sufficient conditions on 
sequences $\{F_n\}$ with {\it complex} coefficients to have equidistribution. In \cite{BLL:hilbert}, existence 
of such sequences is shown for any compact, circled pseudoconvex set in $\C^2$.
\end{remark}

\section{Proof of Theorem~\ref{T:units_equi}}\label{S:units_equi}

This section is devoted to proving Theorem~\ref{T:units_equi}, where our goal is to show that the zeros of
certain sequences of polynomials 
$\{p_n\}$ in $\mathbb{Z}[z]$ are equidistributed with respect to a fixed measure on a compact set $K$ 
whose capacity with respect to 0 and $\infty$ is equal to one.  
The key idea is to apply Theorem~\ref{T:homo_equi} to a lift of the polynomials $p_n$ 
and to a compact circled subset of $\C^2$ derived from $K$. 
We begin with the construction of this set derived from $K$.
\smallskip

Given a non-polar compact subset $K$ of $\C\setminus \{0\}$,
we define a compact circled set $\widehat{K}$ in $\C^2$
whose homogeneous capacity matches the capacity of $K$ with respect to 0 and $\infty$. 
Let $\mathbf{s}=(s_1, s_2)\in \mathscr{P}$ be a probability vector such that 
\begin{equation}\label{E:convex0}
s_1 \lim_{z \to 0} [\gr_K(z, 0) + \log |z|] +s_2 \gr_K(0, \infty)=\robcons(K) 
\end{equation}
and
\begin{equation}\label{E:convexinfty}
s_1\gr_K(\infty, 0)+ s_2 \lim_{z \to \infty} [ \gr_K(z, \infty) - \log |z|]=\robcons(K).
\end{equation}
As discussed in Section~\ref{S:prelim}, there always exists a probability vector
satisfying \eqref{E:convex0} and \eqref{E:convexinfty}, except in the following case.
Let $K$ be such that
$0$ and $\infty$ are in different components of $\sph\setminus K$
(thus $\gr_K(0, \infty)=\gr_K(\infty, 0) =0$) and one of 
$\lim_{z \to 0} [\gr_K(z, 0) + \log |z|]$ and $\lim_{z \to \infty} [ \gr_K(z, \infty) - \log |z|]$ 
is positive while the other is negative. We show that there does not exist a sequence of polynomials
satisfying the hypothesis of Theorem~\ref{T:units_equi}. Thus, in this case, Theorem~\ref{T:units_equi}
is vacuously true. To see this, assume that 
$\lim_{z \to 0} [\gr_K(z, 0) + \log |z|]$ is negative
and $\lim_{z \to \infty} [ \gr_K(z, \infty) - \log |z|]$ 
is positive (otherwise we work with $1/K$). Thus
$\capa (K)<1$. Now for, $p \in \Z[z]$ with $p(0)\not = 0$, recalling from Section \ref{S:intro} that
\[
\hgtone_K(p)=\frac{1}{\deg p} \Big(\log \vert p(0)\vert + \log \vert {\rm lead} (p)\vert + 
\sum_{p(x)=0}[\gr_K (x, 0) +\gr_K (x, \infty)] \Big),
\]
we observe that
\[
\hgtone_K(p)
%=\frac{1}{\deg p} \Big(\log \vert {\rm const} (p)\vert + \log \vert {\rm lead} (p)\vert +
%\sum_{p(x)=0}\gr_K (x, 0) +\gr_K (x, \infty) 
\geq \frac{1}{\deg p} \Big(\log \vert {\rm lead} (p)\vert + 
\sum_{p(x)=0} \gr_K (x, \infty)\Big).
\]
Since $\capa (K)<1$, it follows that there does {\bf not} exist a sequence $\{p_n\}$ in $\Z[z]$
with simple zeros for which
\[
\frac{1}{\deg p_n} \Big(\log \vert {\rm lead} (p_n)\vert + 
\sum_{p_n(x)=0} \gr_K (x, \infty)\Big) \to 0
\]
as $n \to \infty$ (see, for instance, \cite{NM}). Therefore, we conclude that there does not exist a sequence $\{p_n\}$ in $\Z[z]$
with simple zeros for which $\hgtone_K(p_n) \to 0$ as $n \to \infty$.

\smallskip

Given $K$ so that there exists a probability vector
satisfying \eqref{E:convex0} and \eqref{E:convexinfty}, we define the function $f_K$ on $\C^2$ as follows:
\begin{equation}\label{E:Robin}
f_K(z_1, z_2) =
\begin{cases}
s_1 \gr_K\big(\frac{z_1}{z_2}, 0\big) +s_2 \gr_K\big(\frac{z_1}{z_2}, \infty\big)+\log |z_1|^{s_1}|z_2|^{s_2} & \text{if } z_1z_2 \neq 0, \\
\robcons(K) + \log |z_2| & \text{if } z_1 = 0, \\
\robcons(K) + \log |z_1| & \text{if } z_2 = 0.
\end{cases}
\end{equation}
It is easy see that the function $f_K$ is logarithmically homogeneous. We now assume $K$ is regular 
as we will give the proof of Theorem~\ref{T:units_equi} for such $K$; in the appendix we indicate the minor 
modifications in the general non-polar setting. From \eqref{E:convex0} and \eqref{E:convexinfty}, it follows that
$f_K$ is continuous on $\C^2 \setminus \{0,0\}$.
Note that $f_K(0,0) = - \infty$. Since $\gr_K(z,0)$ and $\gr_K(z,\infty)$ are subharmonic,
$f_K$ is plurisubharmonic on $\C^2$.
As a result, the set $\widehat{K} :=\{ f_K \leq 0\}$ is compact, circled and pseudoconvex, 
%the function $f_K$ is the Robin Function of $\widehat{K}$
and $\robin_{\widehat K} \equiv f_K$.
\smallskip

We next see that the push forward of the Monge--Amp\`ere measure $\mu_{\widehat K}$ on $\widehat K$ (recall (\ref{mameas})) 
is the measure of our interest on the set $K$. 

\begin{lemma}\label{L:push}
 $\pi_*(\mu_{\widehat K}) = s_1 \nu_K^0 + s_2\nu_K^{\infty}$. 
\end{lemma}

\begin{proof}
It follows from the definition of $f_K$ that
 $dd^c f_K = \pi^* (s_1 \nu_K^0 + s_2\nu_K^{\infty})$. Therefore,
by Result~\ref{pushf}, we have
$\pi_* (\mu_{\widehat K})= s_1 \nu_K^0 + s_2\nu_K^{\infty}$.
\end{proof}

The support of $\mu_{\widehat K}$ is contained in the set $\{\robin_{\widehat K} =0\}$, equivalently
$\{f_K =0\}$. Note that the support of $\nu_K^0$ is contained in the set $\{\gr_K(z, 0)=0\}$
and the support of $\nu_K^{\infty}$ is contained in the set $\{\gr_K(z, \infty)\}=0$.
Thus, by Lemma~\ref{L:push}, the support of $\mu_{\widehat K}$ is contained in the set
$$\Big\{(z_1, z_2) \in \C^2 :\gr_K\Big(\frac{z_1}{z_2}, 0\Big)= \gr_K\Big(\frac{z_1}{z_2}, \infty\Big)=0\Big\}.$$
Therefore, it follows from the definition of $f_K$ that the support of $\mu_{\widehat K}$ is 
contained in the set $\{|z_1|^{s_1}|z_2|^{s_2}=1$\}. This fact will be useful in proving the next lemma.

\begin{lemma}\label{L:capequal}
$\homcap(\widehat K)= \Cancap(K)$.
\end{lemma}

\begin{proof}
Writing $\mathbf z = (z_1, z_2)$ and $\mathbf w = (w_1, w_2)$, we have
%\[
%I(\mu_{\widehat K})= - \int \int \log \left \vert \frac{z_1}{z_2}- \frac{w_1}{w_2} \right\vert d\mu_{\widehat K}(\mathbf z)d\mu_{\widehat K}(\mathbf w)
%- \int \int \log \vert z_2 w_2 \vert d\mu_{\widehat K}(\mathbf z)d\mu_{\widehat K}(\mathbf w).
%\]
\begin{equation}\label{fiveone}
\I(\mu_{\widehat K})= - \int \int \log \left \vert \frac{z_1}{z_2}- \frac{w_1}{w_2} \right\vert d\mu_{\widehat K}(\mathbf z)d\mu_{\widehat K}(\mathbf w)
- 2 \int \log \vert z_2 \vert d\mu_{\widehat K}(\mathbf z).
\end{equation}
By Lemma~\ref{L:push}, we have $\pi_*(\mu_{\widehat K})= s_1 \nu_K^0 + s_2\nu_K^{\infty}$.
Therefore, the first term on the right side of (\ref{fiveone}) becomes
\begin{align}
-\int \int \log  \left \vert \frac{z_1}{z_2}- \frac{w_1}{w_2} \right \vert 
d\mu_{\widehat K}(\mathbf z)d\mu_{\widehat K}(\mathbf w)
=-s_1^2 &\int \int \log |z-w| d\nu_K^0(z) d\nu_K^{0}(w) \notag \\
-s_2^2 \int \int \log |z-w| d\nu_K^{\infty}(z) d\nu_K^{\infty}(w)
&-2s_1s_2 \int \int \log |z-w| d\nu_K^{\infty}(z) d\nu_K^{0}(w). \notag
\end{align}
Note that by Frostman's theorem, this last double integral simplifies to 
$$- \int \int \log |z-w| d\nu_K^{\infty}(z) d\nu_K^{0}(w) = \int I (\nu_K^{\infty}) d\nu_K^{0}(w)=I (\nu_K^{\infty}).$$
Next, since the support of $\mu_{\widehat K}$ is contained in the set $\{|z_1|^{s_1}|z_2|^{s_2}=1\}$, 
and again using $\pi_*(\mu_{\widehat K})= s_1 \nu_K^0 + s_2\nu_K^{\infty}$, the second term on the right 
side of (\ref{fiveone}) becomes
%\begin{align}
% - \int \int \log \vert z_2 w_2 \vert d\mu_{\widehat K}(\mathbf z)d\mu_{\widehat K}(\mathbf w)
%&= s_1 \int \int \log \left \vert \frac{z_1}{z_2}\cdot \frac{w_1}{w_2} \right \vert d\mu_{\widehat K}(\mathbf z)d\mu_{\widehat K}(\mathbf w) \notag \\
%&= 2s_1 \int \log |z| \, d (s_1 \nu_K^0 + s_2\nu_K^{\infty})(z) \notag \\
%&=s_1^2 \int \int \log |zw| d\nu_K^0(z) d\nu_K^{0}(w)
%+ 2s_1s_2 \int \log |z| d\nu_K^{\infty}(z). \notag
%\end{align}
\begin{align}
 - 2\int \log \vert z_2 \vert d\mu_{\widehat K}(\mathbf z)
&= 2   \int \log \left \vert \frac{z_1}{z_2} \right \vert^{s_1} d\mu_{\widehat K}(\mathbf z) 
-2 \int \log (\vert {z_1}\vert^{s_1} \vert {z_2} \vert ^{s_2}) d\mu_{\widehat K}(\mathbf z) \notag \\
&= 2s_1 \int \log |z| \, d [s_1 \nu_K^0 + s_2\nu_K^{\infty}](z)-0 \notag \\
&=s_1^2 \int \int \log |zw| d\nu_K^0(z) d\nu_K^{0}(w)
+ 2s_1s_2 \int \log |z| d\nu_K^{\infty}(z). \notag
\end{align}
%Now observe that this last integral is equal to {\red shall we mention Fubini?}
Observe that $U^{\delta_0}(z)= -\log |z|$. Thus by Fubini's theorem and the definition of the Green's function 
$\gr(z, \infty)$, it follows that
$$\int \log |z| d\nu_K^{\infty}(z) = -\int U^{\nu_K^{\infty}}(z)\,d \delta_0(z)= \gr_K(0, \infty)- I (\nu_K^{\infty}).$$
Putting these calculations together, we obtain
\begin{align}
\I(\mu_{\widehat K})= -s_1^2\int \int \log \left \vert \frac{z-w}{zw} \right \vert d\nu_K^0(z) & d\nu_K^{0}(w)
-s_2^2 \int \int \log |z-w| d\nu_K^{\infty}(z) d\nu_K^{\infty}(w) \notag\\
&+2 s_1s_2  I (\nu_K^{\infty}) + 2 s_1s_2  (\gr_K(0, \infty)- I (\nu_K^{\infty})).  \notag 
\end{align}
Finally, recall from \eqref{E:greeninfty} and \eqref{E:greenzero} that
\[
-\int \int \log |z-w| d\nu_K^{\infty}(z) d\nu_K^{\infty}(w)= \lim_{z \to \infty} [ \gr_K(z, \infty) - \log |z|],
\]
\[
- \int \int \log \left \vert \frac{z-w}{zw} \right \vert d\nu_K^0(z) d\nu_K^{0}(w) = \lim_{z \to 0} [\gr_K(z, 0) + \log |z|].
\]
Using these together with equations \eqref{E:convex0} and \eqref{E:convexinfty},
we finally obtain $\I(\mu_{\widehat K})= \robcons(K)$. Since $\mu_{\widehat K}$ is a  
homogeneous energy minimizing measure on $\widehat K$, we see that
\[\homcap(\widehat K)= e^{-\I (\mu_{\widehat K})}=  e^{-\robcons(K)}= \Cancap(K).\]
\end{proof}

%Recall from Section \ref{S:intro} that
%\[
%\hgtone_K(p)=\frac{1}{\deg p} \Big(\log \vert {\rm const} (p)\vert + \log \vert {\rm lead} (p)\vert + 
%\sum_{p(x)=0}[\gr_K (x, 0) +\gr_K (x, \infty)] \Big).
%\]
\noindent We are now ready for the proof of the main theorem.

\begin{proof}[The proof of Theorem~\ref{T:units_equi}]
Let $\{p_n\}$ be a sequence satisfying the hypotheses of Theorem~\ref{T:units_equi}.
We aim to show that the corresponding sequence of measures $\{\nu_n\}$ converges to $\nu_K$. 
Let $(s_1, s_2)$ be a probability vector corresponding to $K$
satisfying \eqref{E:convex0} and \eqref{E:convexinfty}.
Consider the mappings $F_n : \mathbb{C}^2 \to \mathbb{C}^2$ defined by
\[
F_n(z_1, z_2) = \Big(z_2^{\deg p_n}\,p_n \Big(\frac{z_1}{z_2} \Big), \ {z_1}^{m_n} z_2^{\deg p_n- m_n}-1\Big),
\]
where $m_n \in \N$ is such that
$m_n/ \deg p_n \to s_1$ as $n \to \infty$.
Thus $(\deg p_n- m_n)/ \deg p_n \to s_2$ as $n \to \infty$.
Since $p_n \in \Z[z]$,
it follows that each $F_n$ is a polynomial mapping with integer coefficients.
Furthermore, observe that $\deg F_n$ is equal to $\deg p_n$ for all $n$. 
\smallskip

We now apply Theorem~\ref{T:homo_equi} to the mappings $\{F_n\}$ and the compact set $\widehat K \subset \C^2$ 
defined as $\widehat K := \{f_K \leq 0\}$, where $f_K$ is as given in \eqref{E:Robin}.
We begin by showing that
$\hgttwo_{\widehat K}(F_n) \to 0$ as $n \to \infty$.
Let $(x_1, x_2) \in \C^2$ be a zero of $F_n$. Note that 
neither $x_1$ nor $x_2$ can be $0$, since ${x_1}^{m_n} x_2^{\deg p_n- m_n}=1$.
As a result, we conclude that $p_n ({x_1}/{x_2})=0$ and
$
x_1^{m_n} x_2^{\deg p_n- m_n} = 1.
$
This implies that $x_1/x_2$ is a zero of $p_n$, and we also have
\begin{equation}\label{E:zero}
x_1^{\frac{m_n}{\deg p_n}} x_2^{\frac{\deg p_n- m_n}{\deg p_n}} = 1.
\end{equation}
Since the polynomials $p_n$ have simple zeros,
observe that the mappings $F_n$ have $\deg F_n$ linearly independent zeros, i.e.,
the mappings $F_n$ are generic.
By hypothesis, $\hgtone_K(p_n) \to 0$ as $n \to \infty$,
so we deduce that
%\[
%\frac{1}{\deg p_n} \Big(\sum_{p_n(x)=0}\gr_K (x, 0) + \sum_{p_n(x)=0}\gr_K (x, \infty) \Big) \to 0  \quad \text{as } n \to \infty.
%\]
%Thus, we also have 
\begin{equation}\label{E:hgtgreen}
\frac{1}{\deg p_n} \sum_{p_n(x)=0}[s_1\gr_K (x, 0) + s_2\gr_K (x, \infty)]  \to 0
\end{equation}
as $n \to \infty$.
Note that if $\zero \in \C^2$ is a zero of $F_n$, then
$c_n\,\zero$ is also a zero of $F_n$ for any $(\deg F_n)^{th}$ root of unity $c_n$.
Recall that we have $\robin_{\widehat K} \equiv f_K$.
Since $m_n/ \deg p_n \to s_1$ and
$(\deg p_n- m_n)/ \deg p_n \to s_2$ as $n \to \infty$,
we conclude from \eqref{E:Robin}, \eqref{E:zero} and \eqref{E:hgtgreen} that
\[
\frac{1}{(\deg F_n)^2}  \sum_{\zero \in F_n^{-1}(0,0)}  \robin_{\widehat K}^+ (\zero)
=\frac{1}{(\deg F_n)^2}  \sum_{\zero \in F_n^{-1}(0,0)} \max \{0, \robin_{\widehat K}(\zero)\} 
\to 0
\]
as $n \to \infty$. We see from (\ref{resdef}) and (\ref{resdef2}) that
\begin{align}
|\res(F_n)| = |\trueres_{\deg p_n, \deg p_n} (p_n(z), z^{m_n})|
&= |({\rm lead} (p_n))^{\deg p_n} \prod_{p_n(x)=0}x^{m_n}| \notag \\
&= |({\rm lead} (p_n))^{\deg p_n-m_n} \cdot (p_n(0))^{m_n}|. \notag
\end{align}
Since $\hgtone_K(p_n) \to 0$ as $n \to \infty$,
it is easy to verify that
%\[
%\frac{1}{(\deg F_n)^2} \log \vert \res(F_n) \vert = \frac{1}{(\deg p_n)^2} 
%\log \vert {\rm lead} (p_n)^{\deg p_n-m_n} \cdot p_n(0)^{m_n}\vert\to 0
%\]
$({1}/{(\deg F_n)^2}) \log \vert \res(F_n) \vert \to 0$ as $n \to \infty$. 
Thus we conclude that $\hgttwo_{\widehat K}(F_n) \to 0$ as $n \to \infty$.
\smallskip

By hypothesis, $\Cancap(K)=1$. 
Therefore, using Lemma~\ref{L:capequal}, we can conclude that the homogeneous capacity of $\widehat K$ is also equal to one.
This allows us to apply Theorem~\ref{T:homo_equi} to $\widehat{K}$ and the sequence $\{F_n\}$, which yields
\[
\mu_n:=\frac{1}{(\deg F_n)^2} \sum_{F_n(\gamma)=(0,0)} \delta_{\gamma} \to \mu_{\widehat{K}}  \quad \text{as } n \to \infty.
\] 
As noted earlier, if $(x_1, x_2)$ is a zero of $F_n$, then $x_1/x_2$ is a zero of $p_n$. Consequently, we have
$\pi_*(\mu_n)=\nu_n$. Now, by Lemma~\ref{L:push}, we also know that
$\pi_*(\mu_{\widehat{K}})=s_1 \nu_K^0 + s_2\nu_K^{\infty}$. Therefore, it immediately follows that
\[
\nu_n=\frac{1}{\deg p_n} \sum_{p_n(x)=0} \delta_{x} \to s_1 \nu_K^0 + s_2\nu_K^{\infty} \quad \text{as } n \to \infty,
\]  
which completes the proof. 
\end{proof}

We conclude this section by providing examples of sets whose capacity with respect to $0$ and $\infty$ is equal to one. 
Specifically, we demonstrate that the preimage of a compact set $E$ with logarithmic capacity one under certain rational maps 
has the same capacity with respect to $0$ and $\infty$. As an illustration, we can choose $E$ to be either a circle of radius one or 
an interval of length four in the complex plane.
To establish this result, we require the following pull-back formula for Green's functions: let $\rat$ denote a rational map with poles at $\{x_i\}$ where $x_i$ has multiplicity $m_i$. Then, the following identity holds (cf., Lemma 2.1.4 \cite{Cantor:edtdsa80}):
\begin{equation}\label{pullb}
\sum_i m_i \gr_{\rat^{-1}(E)}(z, x_i)= \gr_E(\rat(z), \infty).
\end{equation}
Here $\gr_{\rat^{-1}(E)}(z, x_i)$ is the Green's function for $\rat^{-1}(E)$ with pole at $x_i$. We only require this formula when the poles are $0$ and $\infty$.

\begin{proposition}\label{P:capacityone}
Let $E \subset \C$ be a compact set with $\capa(E)=1$. 
Consider a rational map $\rat$ with complex coefficients given by
\[
\rat (z)= \frac{z^n + \dots \pm 1}{z^j},
\]
where $j < n$. Then $\Cancap(\rat^{-1}(E))=1$. 
Moreover, the probability vector 
$(\frac{j}{n}, \frac{n-j}{n})$ satisfies the conditions in \eqref{E:convex0} and \eqref{E:convexinfty}
for the set $\rat^{-1}(E)$.
\end{proposition}

\begin{proof}
Let $K:= \rat^{-1}(E)$, and let $\Gamma (K)$ be the matrix associated with $K$ as described in 
Section~\ref{S:Cantor}.
We show that for the probability vector $(\frac{j}{n}, \frac{n-j}{n})$, the corresponding weighted sums are zero.
It then follows from the discussion in Section~\ref{S:value} that the value of $\Gamma (K)$
is 0. By the pull-back formula (\ref{pullb}) for Green's functions, we have
$$j\gr_K(z, 0)+(n-j) \gr_K(z, \infty)=\gr_E (\rat (z), \infty),$$
which leads to the following expression:
\[
j \lim_{z \to 0} [\gr_K(z, 0) + \log |z|] +(n-j) \gr_K(0, \infty)= \lim_{z \to 0}[\gr_E (\rat (z), \infty)+ j \log |z|].
\]
Note that $\log |\rat (z)|= \log |z^n + \dots \pm 1|- j \log |z|$. Thus, we obtain
\[
\lim_{z \to 0} [\gr_E (\rat (z), \infty)+ j \log |z|]= \lim_{z \to 0}[ \gr_E (\rat (z), \infty)- \log |\rat (z)| + \log |z^n + \dots \pm 1|]=0.
\]
The last equality follows because the logarithmic capacity
of the set $E$ is equal to one. Similarly, we have
\[
j \gr_K(\infty, 0) + (n-j) \lim_{z \to \infty} [ \gr_K(z, \infty) - \log |z|]= \lim_{z \to \infty} [\gr_E (\rat (z), \infty)- (n-j) \log |z|].
\]
Since $n >j$, observe that $\log |\rat (z)|= (n-j) \log |z| + o(1)$ as $z \to \infty$. Therefore
\[
\lim_{z \to \infty} [\gr_E (\rat (z), \infty)- (n-j) \log |z|]= \lim_{z \to \infty} [\gr_E (\rat (z), \infty)- \log |\rat(z)|]=0.
\]
Hence, for the probability vector $(s_1,s_2)=(\frac{j}{n}, \frac{n-j}{n})$, the corresponding weighted sums in  \eqref{E:convex0} and \eqref{E:convexinfty} are zero,
and we conclude that $\Cancap(\rat^{-1}(E))=1$.
\end{proof}

Note that for $K:=\rat^{-1}(E)$, where $\rat$ and $E$ are as in the above proposition, the limiting measure $\nu_K$
in the statement of Theorem~\ref{T:units_equi} is 
$\frac{j}{n} \nu_K^0 + \frac{n-j}{n}\nu_K^{\infty}$.

%{\blue \begin{example}
%What if $K$ contains either $0$ or $\infty$?
%\end{example}

\begin{remark} 
One can define the capacity of a set $K$ with respect to points $a,b\in \C$ lying outside of $K$ by considering the matrix
\[
\Gamma_{a,b} (K):=
\begin{bmatrix}
\displaystyle \lim_{z \to a} [\gr_K(z, a) + \log |z-a|] & \gr_K(a, b) \\
\gr_K(b, a) & \displaystyle \lim_{z \to b} [ \gr_K(z, b) + \log |z-b|]
\end{bmatrix}.
\]
to define the Robin's constant of $K$ with respect to $a,b$ as $\gamma_{a,b}(K):=\textnormal{val}(\Gamma_{a,b} (K))$; then 
the Cantor capacity of $K$ with respect to $a$ and $b$ is ${\rm cap}_{a,b}(K):=e^{-\gamma_{a,b}(K)}$.
However, given a set $K$ with ${\rm cap}_{a,b}(K)=1$, it is unclear how to define $\widehat K \subset \C^2$
so that $\pi(\widehat K)=K$ and $\homcap (\widehat K) =1$. Moreover, it is unclear how to prove an equidistribution result in this setting using Theorem~\ref{T:homo_equi} as one also needs to be able to lift 
univariate polynomials $\{p_n\}$ to polynomial mappings on $\C^2$ in an appropriate fashion. 
\end{remark}

\section{Appendix: Theorems~\ref{T:units_equi} and ~\ref{T:homo_equi}, general case}\label{S:nonregular}

To generalize Theorem~\ref{T:units_equi} to nonpolar compacta $K$ and Theorem~\ref{T:homo_equi} to compact, circled nonpluripolar 
$\nice$, recall that for $K$ regular, $\widehat{K}$ was defined in Section \ref{S:units_equi} as $\widehat{K}:=\{f_K \leq 0\}$. 
In the general nonpolar case, since this set may not be closed, we take $\widehat{K}$ to be the closure of the set $\{f_K \leq 0\}$. 
In this case, $\widehat{K}$ is a subset of the union of $\{f_K \leq 0\}$ and the set
of discontinuities of $f_K$. The function $f_K$ is discontinuous at
points $(z_1, z_2)$ such that $\gr_K(\cdot,0)$ or $\gr_K(\cdot ,\infty)$ are discontinuous at the point $z_1/z_2$. 
Since the set of these discontinuities in $\C$ is a polar set, the set
of discontinuities of $f_K$ is a pluripolar set. Therefore, the pluricomplex
Green's functions and hence the Robin functions for $\widehat{K}$ and $\{f_K \leq 0\}$ coincide (see \cite[Corollary 5.2.5]{Klimek91}).
Moreover, since $f_K$ is the Robin function of the Borel set $\{f_K \leq 0\}$, we conclude that
$\robin_{\widehat K} \equiv f_K$. Thus Theorem~\ref{T:units_equi} generalizes to nonpolar compacta $K$ provided we can show 
that Theorem~\ref{T:homo_equi} generalizes to compact, circled nonpluripolar $\nice$.
\smallskip

To this end, we first remark that (\ref{E:infsup})
$$-\infty< \inf_{\Vert \mathbf z \Vert=1} \robin_{\nice}(\mathbf z) \leq \sup_{\Vert \mathbf z \Vert=1} \robin_{\nice}(\mathbf z) <\infty$$
still holds in this setting; i.e., only non-pluripolarity of $\nice$ is needed. The only other items to check are that Results \ref{pushf} 
and \ref{R:unique} remain valid. For the former, the proof in \cite{DeM:dorm03} is still valid as for $\zeta \in \sph$, while 
the measure $dd^cV_{\nice}^*|_{\pi^{-1}(\zeta)}$ is normalized Lebesgue measure 
$m_{\zeta}$ on $\partial \nice \cap \pi^{-1}(\zeta)$ only when this set is a circle, this occurs for a.e. $\zeta \in \sph$, 
which suffices for the conclusion. 
\smallskip

For Result \ref{R:unique} we must show for any compact, circled and nonpluripolar $\nice \subset \C^2$, 
the measure $\mu_\nice$
is the unique $S^1$-invariant homogeneous energy minimizing measure for $\nice$. 
To modify the argument in \cite[Theorem 3.1]{DeM:dorm03}, we first observe 
that an energy minimizing measure $\mu$ is supported in the set $\nice \setminus \{\robin_{\nice} < 0\}$.
The sets $\nice \setminus \{\robin_{\nice} < 0\}$ and $\{\robin_{\nice} = 0\}$ differ
by a pluripolar set and thus $\robin_{\nice}\equiv 0$ on ${{\rm supp}\, \mu}$
except possibly on a pluripolar set $\ppolar$. It follows by logarithmic homogeneity of $\U^{\mu}$ that 
$$\U^{\mu} \leq -\I(\mu) +\robin_{\nice}$$ 
on $\C^2$. To show the reverse inequality, for $c\in \C$, let $L_c=\{(c,z_2): z_2 \in \C\}$ be the complex 
line $z_1=c$. It is easy to see that 
$$C:=\{c\in \C: L_c \cap \ppolar \ \hbox{is not polar in} \ L_c\}$$ 
is a polar set in $\C$. For $L_c$ with $c\not \in C$, 
we let $S_c:=\hbox{supp}\ dd^c \U^{\mu}|_{L_c}$ and $v_c:=\U^{\mu}|_{L_c} + \I(\mu)$. 
Then the subharmonic function $u_c:= \robin_{\nice}|_{L_c}-v_c$ on $L_c \setminus S_c$ satisfies 
$$\limsup_{(c,z_2) \to (c,\zeta)}u_c(c,z_2) \leq 0 \ \hbox{for} \ (c,\zeta) \in S_c \setminus E_c$$ 
where $E_c$ is polar. By the extended maximum principle, $u_c  \leq 0$ on $L_c \setminus S_c$. 
Since this inequality remains true on $S_c$, we have $$\U^{\mu}\geq \robin_{\nice} -\I(\mu)$$ 
on all $L_c$ with $c \not \in C$. The union of the collection of lines $L_c$ with $c\in C$ is pluripolar and hence has measure $0$; thus  
$\U^{\mu}\geq \robin_{\nice} -\I(\mu)$
on all of $\C^2$.

\end{document}